\documentclass[11pt]{amsart}
\usepackage{latexsym, amsmath,amssymb,amsthm,ifthen,color,amsaddr,a4wide}
\usepackage[all]{xypic}
\usepackage[mathcal]{euscript}
\usepackage{mathrsfs}
\usepackage{tikz}

\newtheorem{thm}{Theorem}[section]
\newtheorem{corol}[thm]{Corollary}

\newtheorem{prop}[thm]{Proposition}
\newtheorem{defin}[thm]{Definition}

\theoremstyle{remark}
\newtheorem{rem}[thm]{Remark}
\newtheorem{ex}[thm]{Example}
\newenvironment{remark}{\begin{rem}\rm}{\end{rem}}

\newcommand{\cN}{\mathcal{N}}

\newcommand{\shs}{\hspace{.2cm}}

 \newcommand{\su}{\operatorname{\mathbf{su}}}
\def\bS{{\boldmath \Sigma}}

\newcommand{\PP}{{\mathbb P}}

\newcommand{\C}{{\mathbb C}}
\newcommand{\Z}{{\mathbb Z}}

\newcommand{\g}{{\mathfrak g}}
\newcommand{\T}{{\mathfrak a}}
\newcommand{\D}{{\mathfrak d}}
\newcommand{\E}{{\mathfrak e}}

\newcommand{\qee}{\mbox{\hspace{0.2mm}}\hfill$\triangle$}

\newcommand{\marginnote}[1]{\ifthenelse{\isodd{\thepage}}{\normalmarginpar}
{\reversemarginpar}\marginpar{\fbox{\parbox{28mm}{\sloppy\footnotesize #1}}}}
\begin{document}
\bigskip
\title{Geometry and Topology of String Junctions}
\author{\small Antonella Grassi$^\P$, James Halverson $^{\S}$ and Julius L. Shaneson$^\P$}
\address{\vskip-10pt\rm $^\P$ Department of Mathematics, University of Pennsylvania,\\
David Rittenhouse Laboratory, 209 S 33rd Street,\\ Philadelphia, PA 19104, USA% \footnote{Support for this work was provided by the NSF Research Training Group Grant
% DMS-0636606, by {\sc prin} ``Geometria delle variet\`a  algebriche
% e dei loro spazi dei moduli '' and  the {\sc infn} project {\sc pi14} ``Nonperturbative dynamics of gauge theories''. U.B. is a member of the {\sc vbac} group.}
}
\address{\rm $^\S$ Kavli Institute for Theoretical Physics \\
University of California
Santa Barbara, CA 93101, USA}
\address{\rm $^\S$ Department of Physics \\
Northeastern University
Boston, MA 02115, USA}
\thanks{E-mail: {\tt grassi@sas.upenn.edu, j.halverson@neu.edu, shaneson@math.upenn.edu}}
\date{\today}
\begin{abstract}
We study elliptic fibrations by analyzing suitable deformations of the fibrations and vanishing cycles.  We introduce geometric string junctions and describe some of their properties. We show how  the geometric string junctions   manifest the  structure of the  Lie algebra of  the Dynkin diagrams  associated  to the singularities of the elliptic fibration. One application  in physics is in F-theory, where our novel approach connecting deformations and Lie algebras describes the structure of generalized type IIB seven-branes and string junction states which end on them.
\end{abstract}

\maketitle

\vspace{-.4cm}
\section{Introduction}
An elliptic fibration is  a morphism $\pi:X\rightarrow B$ such that $\pi^{-1}(p)= E_p$ for a general point $p\in B$ is a smooth elliptic elliptic curve\ (a torus $T^2$); the discriminant locus is
 $\Sigma = \{q \in B\, \text{ such that }
\pi^{-1}(q)\ne T^2 \}$.
In this paper we take $X$ and $B$ to be smooth;  if  $\pi$ is also assumed to have a section $\sigma$, $X$ is the (smooth) resolution of  the Weierstrass model $W$ of the
fibration
\cite{Na88} with Gorenstein singularities. $W$ is defined by the Weierstrass equation:
$ y^2z-(x^3+f xz^2+g z^3) = 0, $ where
$f, g$ are sections of appropriate bundles on $B$.
If $\dim W=2$,  the singularities are the rational double points. It was noted by Du Val and Coxeter \cite{DV, coxeter:weyl}
that rational double points
are classified by the Dynkin diagrams of the
simply laced Lie algebras $\g$ of type $\T_n, \D_n, \E_6, \E_7, \E_8$:
 in fact,
the dual diagram of the exceptional divisors in the minimal resolution
is one of the above Dynkin diagrams. In the case of higher dimensional elliptic fibrations also non simply-laced Dynkin diagram can occur.

We study elliptic fibrations by analyzing suitable deformations of the fibrations; we introduce ``geometric string junctions".
String junctions  were defined in the physics literature by DeWolfe, Gaberdiel and Zwiebach  \cite{GaberdielZwiebach, DeWolfeZwiebach} as equivalence relations of closed paths in a punctured disc $\C \setminus \bS$; the disc is the base of an elliptic fibration   with discriminant locus $\bS$. The authors assign composition rules  and  show that the junctions reflect the structure of exceptional gauge algebras of the elliptic fibration. The gauge algebras which arise in this construction are the simply laced ones $\T, \D, \E$. Bonora and Savelli \cite{BonoraSavelli}  later derived some non-simply laced algebras from junctions; their construction is algebraic and not expressed in terms of the geometry of higher dimensional elliptic fibrations. We will do this later in the paper.
The techniques presented in this paper have a number of applications in physics; for example, in describing BPS states of $d=4$ $\cN=2$ supersymmetric gauge theories, or in F-theory where they provide a direct approach to the analysis of generalized type IIB seven-branes and the string junction states which end on them
 \cite{V1,MV, MVII,DouglasParkSchnell}.

In Section \ref{geometricstrings}  we consider a smooth elliptic surface  on the open  unit disc  $U$
 with  nodal singular fibers over a collection of points $\{q_j\}$.  We then consider suitable, disjoint,  paths  in $U$ from a base point to each $q_j$ (a junction, in the physics language) and the corresponding vanishing cycles and construct thimbles, the prongs in the physics language, in  the relative homology.
A  general \emph{geometric junction} $J$ defines then a chain with boundary in the elliptic fiber over the base point $p$.  Following the physics literature we define the asymptotic charge $a(J) \in H_1(E_p,\Z)$,
 $a(J) = \partial [J]_p \in H_1(E_p)$.
We show that the junctions with asymptotic charge zero are  the images of spherical classes  in $[J] \in H_2(X)$ (Theorems \ref{jct}, \ref{emb}).
We then define a self-intersection product in the space of junctions,  and  we show that if $a(J) = 0$ the topological intersection is equal to the the self-intersection product (Theorem \ref{int}).
We also define an intersection pairing
$\langle J, K \rangle $, which we show it coincides with the topological pairing for junctions of zero asymptotic charge. We provide an explicit formula in terms of the $J_i$ and the intersection numbers of the vanishing cycles.
 If $X=W$ is a Weierstrass model we also provide an explicit alternative descriptions of the class of junctions, which is implemented in a computer program in \cite{sage}, \cite{JimPackage}.
In Section \ref{DefAndMore} we consider a smooth elliptic surface in Weierstrass equation over a disc, with a unique singular fiber over the origin, an ADE singularity.   Klein showed that  resolutions and deformation of ADE surface singularities (also known as {kleinan} singularities) are diffeomorphic.
 We  deform the elliptic fibration to a smooth fibration with nodal fibers, namely we perform a complete Higgsing of the Weierstrass model.  We  study the junctions in the deformed model and  we  prove that  the lattice structure found in the previous Section \ref{geometricstrings} provides the weight structure of the $\T, \D, \E$ algebras associated to the Dynkin diagram of the original singularities. As  a particular case,  we obtain the roots of the $\T, \D, \E$ central singularity and the associated Cartan matrix from the junctions with asymptotic charge zero.  Our deformation analysis of the surface case show that the junctions with a fixed  non-zero asymptotic charge are associated to weights of other representations and all  the possible weights occur; we show that in higher dimensional variety these weights can become associated to junctions of non-zero asymptotic  charge and assume geometric meaning, they become massless in the physics language \cite{GrassiHalversonShaneson2013, GrassiHalversonShaneson2014}.  In contrast analysis of the resolved surface provides only the root structure. 
 
 In addition we show that the deformation analysis distinguish the Kodaira type $III$ (two tangent rational curves) from $I_2$ (a cycle of two rational curves), which are associated to the $\mathfrak{su}(2)$ gauge algebra, and the Kodaira type $IV$ (three rational curves meeting at point) from $I_3$ (a cycle of three rational curves, which are associated to the $\mathfrak{su}(3)$ gauge algebra. This reflects in physics the different brane structure of the two singularities. These results were first presented in our previous papers \cite{GrassiHalversonShaneson2014, GrassiHalversonShaneson2013}, and were obtained with the aid of a computer package especially developed \cite{JimPackage}.
In \cite{GrassiHalversonShaneson2014}  we show  also that the local deformation techniques  of the string junction analysis can be adapted in compact cases,  even in cases when global deformation do not exist and also in higher dimension.
In Section \ref{G2} we revisit the example of the ${\mathfrak g}_2$ algebra first presented in \cite{GrassiHalversonShaneson2013}, and elliptic fibration of
threefolds. 

The techniques developed in Section \ref{geometricstrings} do not assume the existence of a section of the fibration, and in principle can be applied also in that case. Resolutions to a smooth model with trivial canonical class and an equidimensional fibration may  not  be available in higher dimension, minimal models can have terminal singularities,  however  the deformation analysis can still be performed. We will address such situations in a upcoming paper \cite{GrassiHalversonShaneson:TerminalBraids}.   Our techniques can also be extended to other type of fibrations between  varieties which are not necessarily algebraic, for example on varieties with $G_2$ holonomy,  

\vspace{.1cm}
 \noindent{\bf Acknowledgement.} We thank P. Aluffi and M. Esole for
 organizing the Spring 2014 AMS Special Session ``Singularities and Physics". We also are grateful to the referee for her/his useful comments.
   A.G. was in part supported by the NSF
 Research Training Group Grant GMS-0636606. The
 research of J.H. was supported in part by the National Science
 Foundation under Grant No. PHY11-25915.  J.L.S. is supported by
 DARPA, fund no. 553700 and is the Class of 1939 Professor in the
 School of Arts and Sciences of the University of Pennsylvania and
 gratefully acknowledges the generosity of the Class of 1939.  We also thank the Simons
 Center for Geometry and Physics.

\section{Geometric String Junctions}\label{geometricstrings}

Consider the smooth elliptic fibration $\pi:X\longrightarrow U$ over
an open set $U\subset \C$ with $I_1$ singular fibers above $\Sigma =
\{q_1\, \dots, q_N \}$. Fix a base point $p\in B\setminus \Sigma$ with
$E_p$ the elliptic fiber $\pi^{-1}(p)$. Choose a set of continuous embedded paths $\delta_1,...,
\delta_N\,,$ assumed disjoint except for the common starting point $\delta_j(0) = p\,,$ ending at the points $\delta_j(1) = q_j\,.$ Also assume the order is such that the points in which the paths meet a small circle around $p$ go around counter clockwise; for example for a small r, $\delta_j (r) = re^{2\pi i\frac{ (j-1)}{N}}$ in  suitable local coordinates around $q$ as the origin. We also assume that for $\epsilon$ a small
real number, there is also the formula $\delta_j(1-\epsilon) = \epsilon\, e^{2\pi
i\theta_j}\,,$ $\theta_j$ an "angle" in suitable local coordinates around $q_j$ as the origin. (The local co-ordinates around $q_j$ could of course be chosen so $\theta_j = 0\,,$ but below we will use other paths coming in to $q_j$ at different angles.)

%drawing #-5 here%

\vskip 0.2in

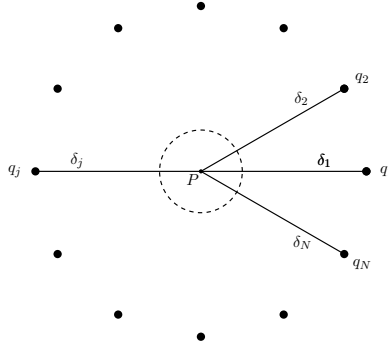
\begin{center}
\scalebox{.55}{
  \begin{tikzpicture}
    \fill[thick] (0:40mm) circle (1mm);
    \node at (0:45mm) {$q_1$};
    \draw[thick] (0,0) -- (0:40mm);
    \node at (5:30mm) {$\delta_1$};
    \fill[thick] (30:40mm) circle (1mm);
    \node at (30:45mm) {$q_2$};
    \draw[thick] (0,0) -- (30:40mm);
    \node at (36:30mm) {$\delta_2$};
    \fill[thick] (-30:40mm) circle (1mm);
    \node at (-30:45mm) {$q_N$};
    \draw[thick] (0,0) -- (-30:40mm);
    \node at (-35:30mm) {$\delta_N$};
    \fill[thick] (0:40mm) circle (1mm);
    \node at (180:45mm) {$q_j$};
    \node at (5:30mm) {$\delta_1$};
    \draw[thick] (0,0) -- (180:40mm);
    \node at (175:30mm) {$\delta_j$};
    \fill[thick] (30:40mm) circle (1mm);
    \fill[thick] (180:40mm) circle (1mm);
    \fill[thick] (60:40mm) circle (1mm);
    \fill[thick] (90:40mm) circle (1mm);
    \fill[thick] (120:40mm) circle (1mm);
    \fill[thick] (150:40mm) circle (1mm);
    \fill[thick] (210:40mm) circle (1mm);
    \fill[thick] (240:40mm) circle (1mm);
    \fill[thick] (270:40mm) circle (1mm);
    \fill[thick] (300:40mm) circle (1mm);
    \draw[thick,dashed] (0mm,0mm) circle (10mm);
    \node at (-2mm,-2mm) {$P$};BaPeVV84
    \fill[thick] (0,0) circle (.5mm);
\end{tikzpicture}
}
\end{center}

Any smooth fiber bundle $Y \to S^1$ over a circle is given by a monodromy of the fiber $F\,.$ This means there is a diffeomorphism \cite{Steenrod1957}.  In our case, over the circle of radius $\epsilon$ around $q_j,,$ the corresponding diffeomorphism $\psi_{j,1-\epsilon}: E_{\delta_j(1-\epsilon)} \to E_{\delta_j(1-\epsilon)}$ will be referred to as the {\it local monodromy} around $q_j\,.$

The assumption that the singular fibers are of type  $I_1$  means that there is a real curve $C_{j,1-\epsilon}$ on the fiber $E_{\delta_j(1-\epsilon)}\,,$ $\epsilon$ small, which the  local monodromy $\psi_{j,1-\epsilon}$ fixes, where:
 $$\psi_{j,1-\epsilon}: E_{\delta_j(1-\epsilon)} \to E_{\delta_j(1-\epsilon)}.$$
  This curve collapses to a point $\hat q_j$ as $\epsilon \to 0\,;$ $\hat q_j$ is the (nodal) singular point in the singular fiber $\pi^{-1}(q_j)\,.$ With any choice of orientation for this curve, the map induced on first homology, ${(\psi_{j,1-\epsilon})}_* : H_1(E_{\delta_j(1-\epsilon)}) \to H_1(E_{\delta_j(1-\epsilon)})\,,$ satisfies \cite{Kodaira}, \cite{BaPeVV84}.

\begin{equation}\label{pl} {(\psi_{j,1-\epsilon})}_*(x) = x - (x\cdot [C_{j,1-\epsilon}])\,[C_{j,1-\epsilon}]\qquad, \qquad x \in H_1(E_{\delta_j(1-\epsilon)})\,.\end{equation}
The equation is a special case of the Picard-Lefshetz formula for this situation. Here
$x\cdot [C_{j,1-\epsilon}]$ is the intersection number of $x$ with the homology class
$[C_{j,1-\epsilon}]$ of the curve $C_{j,1-\epsilon}\,.$

Fix a small $\epsilon_0\,.$ The fibration $\pi$ is trivial over the contractible set $\delta_j([0,1-\epsilon_0])\,;$ let

\begin{equation} \Psi_j: [0,1-\epsilon_0] \times E_{\delta_j(1-\epsilon_0)} \to \pi^{-1}(\delta_j([0,1-\epsilon_0]))\end{equation}
be a trivialization with $\Psi_j(1-\epsilon_0,z) = z\,.$ Then we define the vanishing cycle $\gamma_j \in H_1(E_p)$ as
\begin{equation} \gamma_j = {(\Psi_j|\{0\}\times E_{\delta_j(1-\epsilon_0)})}_*[C_{j,1-\epsilon_0}]\,.\end{equation}
This is the same as the homology class $[C_j]$ of the curve $C_ j = \Psi_j|\{0\}\times E_{\delta_j(1-\epsilon_0)}(C_{j,1-\epsilon_0}) \subset E_p$ that is is the image of the curve $C_{j,1-\epsilon_0} \subset E_{\delta_j(1-\epsilon_0)}\,,$ and we also set
\begin{equation}C_{j,t} = \Psi_j|\{t\}\times E_{\delta_j(1-\epsilon_0)}(C_{j,1-\epsilon_0})\,,\end{equation}
 so $C_j = C_{j,0}\,.$
The class $\gamma_j$ is only defined up to sign, but we will systematically suppress this ambiguity (however see Corollary \ref{alg} below). Finally, we can use the diffeomorphism $\Psi_j|\{0\}\times E_{\delta_j(1-\epsilon_0)}$ of $E_{\delta_j(1-\epsilon_0)}$ to $E_p$ to tranfer the local monodromy at $q_j$ to a {\it global monodromy} $\psi_i : E_p \to E_p$ that fixes $C_j$ and also satisfies the Picard-Lefshetz formula \ref{pl}. The topological type of the fibration is determined by the isotopy classes of the global monodromies \cite{GrassiHalversonShaneson:Math}.

We define the ``prong" (physics terminology) or ``thimble" (symplectic geometry terminology)
\begin{equation} \Gamma_j = \Psi_j([0,1-\epsilon_0]\times C_{j,1-\epsilon_0})\cup \bigcup_{0 <  \epsilon \le\epsilon_0}C_{j,1-\epsilon} \cup \{\hat q_j\}\,,\end{equation}
and more generally we will need to use
\begin{equation} \Gamma_{j,t} = \Psi_j([t,1-\epsilon_0]\times C_{j,1-\epsilon_0})\cup \bigcup_{0 <  \epsilon \le\epsilon_0}C_{j,1-\epsilon} \cup \{\hat q_j\}\,.\end{equation}
The prong is a disk and represents a class $[\Gamma_j] \in H_2(X,E_p)$ with $\partial [\Gamma_j] = \gamma_j\,.$
\vskip 0.2in

%drawing -1 here%
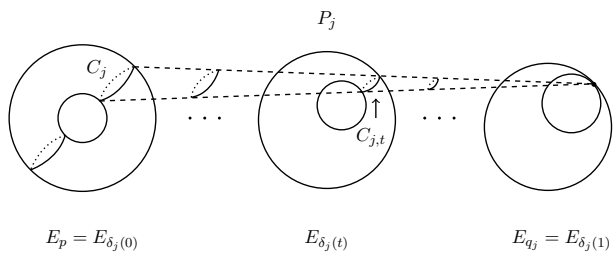
\begin{center}
\scalebox{.65}{
  \begin{tikzpicture}
    % LEFT TORUS
    \draw[thick] (-30mm,0mm) circle (5mm);
    \draw[thick] (-30mm,0mm) circle (15mm);
    \draw[thick,xshift=-30mm] (45:5mm) .. controls (30:10mm) and (45:15mm) .. (45:15mm);
    \draw[thick,xshift=-30mm,dotted] (45:5mm) .. controls (60:10mm) and (45:15mm) .. (45:15mm);
    \draw[thick,xshift=-30mm,] (225:5mm) .. controls (225+15:10mm) and (225:15mm) .. (225:15mm);
    \draw[thick,xshift=-30mm,dotted] (225:5mm) .. controls (225-15:10mm) and (225:15mm) .. (225:15mm);
    % LEFT IN BETWEEN
    \draw[thick,xshift=-12mm] (45:6mm) .. controls (30:10mm) and (45:14mm) .. (45:14mm);
    \draw[thick,xshift=-12mm,dotted] (45:6mm) .. controls (60:10mm) and (45:14mm) .. (45:14mm);
    \fill[thick,xshift=-5mm] (-3mm,0mm) circle (.3mm);
    \fill[thick,xshift=-5mm] (0mm,0mm) circle (.3mm);
    \fill[thick,xshift=-5mm] (3mm,0mm) circle (.3mm);
    % MIDDLE TORUS
    \draw[thick,xshift=23mm,yshift=2.6mm] (0mm,0mm) circle (5mm);
    \draw[thick,xshift=23mm,yshift=2.6mm] (-3mm,-3mm) circle (14mm);
    \draw[thick,xshift=23mm,yshift=.6mm] (50:6.1mm) .. controls (35:9mm) and (45:11mm) .. (45:11mm);
    \draw[thick,xshift=23mm,yshift=.6mm,dotted] (50:6.1mm) .. controls (55:9mm) and (45:11mm) .. (45:11mm);
    % RIGHT IN BETWEEN
    \draw[thick,xshift=35mm,dotted,yshift=.2mm] (45:8mm) .. controls (52:9.5mm) and (45:11mm) .. (45:11mm);
    \draw[thick,xshift=35mm,yshift=.2mm] (45:8mm) .. controls (38:9.5mm) and (45:11mm) .. (45:11mm);
    \fill[thick,xshift=43mm] (-3mm,0mm) circle (.3mm);
    \fill[thick,xshift=43mm] (0mm,0mm) circle (.3mm);
    \fill[thick,xshift=43mm] (3mm,0mm) circle (.3mm);
    % RIGHT TORUS
    \draw[thick,xshift=70mm,yshift=3mm] (0mm,0mm) circle (6mm);
    \draw[thick,xshift=70mm,yshift=3mm] (-4.8mm,-4.8mm) circle (13mm);
    \fill[xshift=70mm] (4.6mm,6.9mm) circle (.5mm);
    % LONG LINES
    \draw[thick,dashed] (-19.5mm,10.5mm) -- (74.5mm,7mm);
    \draw[thick,dashed] (-26.5mm,3.5mm) -- (74.5mm,7mm);
    % LABELS
    \node at (-28mm,-25mm) {$E_p=E_{\delta_j(0)}$};
    \node at (-27mm,10mm) {$C_j$};
    \node at (29mm,-4mm) {$C_{j,t}$};
    \draw[thick,->](30mm,0mm) -- (30mm,4mm);
    \node at (20mm,-25mm) {$E_{\delta_j(t)}$};
    \node at (20mm,20mm) {$P_j$};
    \node at (68mm,-25mm) {$E_{q_j}=E_{\delta_j(1)}$};
  \end{tikzpicture}
  }
\end{center}

We have the following alternate description when there is a Weierstrass equation:

\begin{prop}\cite{GrassiHalversonShaneson:Math}\label{roots} Let $X = W$ have the Weierstrass equation

\begin{equation}
    z y^2 = x^3 + f\, xz^2 + g\,z^3
  \end{equation}
  with section $\sigma$.
Then $E_q - {\sigma(q)}$ is the two-fold branched cover of $\C$ branched at the roots of
$x^3 + f(q)x + g(q) = 0\,.$ For $0 \le t \le 1\,,$ let $\rho_{j,1}(t)$ and $\rho_{j,2}(t)$ be continuous paths of two of the roots at $ \delta_j(t)\,,$ with the property $\rho_{j,1}(1) = \rho_{j,2}(1)$ at the singular point $q_j\,$ Let $\rho_{j,3}(t)$ be the path of the remaining root, and assume that for all $0 \le s,t \le 1\,,$ $\rho_{j,3}(s) \ne \rho_{j,1}(t)$ and   $\rho_{j,3}(s) \ne \rho_{j,2}(t)\,.$ Let $\bar C_{j,t}$ be the closed loop (not necessarily embedded) in $E_{\delta_j(t)}$ that lies over the path $ \rho_{j,t} = \{\rho_{j,1}(s), \rho_{j,2}(s)| t \le s \le 1\}\,.$ Then these loops have a consistent orientations so that if $\bar \Gamma_j = \bigcup_{0\le t \le 1}\bar C_{j,t}\,,$ then
$[\bar \Gamma_j] = [\Gamma_j]\,;$ in particular $\partial [\bar \Gamma_j] = \gamma_j\,.$
\end{prop}

The proof of the above proposition, presented in \cite{GrassiHalversonShaneson:Math},  also provides the following algorithm for determining vanishing cycles:

\begin{corol}\label{alg} Assume, in addition to the hypotheses of Proposition \ref{roots} (for simplicity) the roots $\rho_1,\rho_2,\rho_3$ of $x^3 + f(p)x + g(p) = 0\,,$ the first two being the ones that merge at $q_j\,,$ are not on a common (real) line. Let $m_1$ be the number of times $ \bar \rho_{j,0}$ crosses from one side of the interior of the straight line joining
$\rho_1= \rho_{j,1}(0)$ and $\rho_2 = \rho_{j,1}(0)$ to the other. Let $m_2$ be one half the sum of the intersection numbers of the path $\rho_{j,0}$ (with either orientation) and this straight line at the endpoints. (If an intersection at an endpoint is not transverse, make it so by a small perturbation and count any additional crossings of the line that this produces in $m_1$ as well.) Let $Z_1,Z_2,Z_3 \in H_1(E_p)$ be represented by loops that are inverse images in $E_p$ of straight lines joining $\rho_1$ and $\rho_2\,,$ $\rho_2$ and $\rho_3\,,$ and $\rho_3$ and $\rho_1\,.$ If $m_1 + m_2$ is even then $\gamma_j = \pm Z_1\,.$ If $m_1 + m_2$ is odd, then $\gamma_j = \pm Z_2 \pm Z_3\,,$ with any choice of signs so that $\gamma_j\cdot Z_1 \ne 0\,.$
\end{corol}

In \cite{GrassiHalversonShaneson2014, GrassiHalversonShaneson2013} and in the examples in this paper we apply
 Proposition  \ref{roots} and the algorithm in  simple special cases.

For the rest of this section we do not assume the existence of a Weierstrass model. Nevertheless, it can be shown \cite{GrassiHalversonShaneson:Math} that there exists a \emph{topological} section $\sigma: U \to X\,.$ We will not give a proof here, but the reason is that the global monodromy maps, which determine the topological type of the fibration, are isotopic to maps with a common fixed point.

\vskip 0.2in

\begin{defin} As above, consider the smooth elliptic fibration $\pi:X\longrightarrow U$ over
an open set $U\subset \C$ with $I_1$ singular fibers above $\Sigma =
\{q_1\, \dots, q_N \}$ and corresponding vanishing cycles $\{\gamma_1, \dots, \gamma_N\}$. Fix a base point $p\in B\setminus \Sigma$ with
$E_p=\pi^{-1}(p)$. \\
$J=(J_1,\dots,J_N)\in\Z^N$
is a {\bf junction} and the cycle $a_p(J) = a(J)=\sum_{1}^N J_i \gamma_i \in H_1(E_p,\Z)$ is
its {\bf asymptotic charge}.
\end{defin}

\begin{remark}A junction defines a chain (actually the image of a union of 2-disks) $\sum_1^N J_j\Gamma_j$ or $\sum_1^N J_j\bar \Gamma_j$ in $X\,,$ and hence a homology class

\begin{equation} [J]_p = \sum_1^NJ_j[\Gamma_j] \in H_2(X,E_p)\,,\end{equation}
(This homology class actually depends on the order of singular points, up to a cylcic permutation of order $N\,.$.)

Clearly $a(J) = \partial [J]_p \in H_1(E_p)\,;$ hence if $a(J) = 0\,,$ $[J]_p$ will be the image of a class in $[J] \in H_2(X)\,;$  it is only well defined up to a multiple of the image of orientation class $[E_p]$ of $E_p$ in $H_2(X)\,.$  It will be unique subject to the extra condition that $[J]\cdot \sigma(U) = 0\,;$ this intersection number is well defined because the image of the section is a proper submanifold. If $a(J) = 0\,,$ then from the explicit construction one can see that $[J]$ is spherical, i.e. represented by a map $S^2 \to X\,.$ The class $[J]$ also depends on the basepoint $p$ and the choice of paths.
\end{remark}

\begin{thm}\label{jct} Let $\bf J$ denote the abelian group  of junctions. Assume $U$ is the interior of a region bounded by a closed embedded smooth curve. Then $J \mapsto [J]_p= \sum_1^NJ_j[\Gamma_j] \in H_2(X,E_p)$ induces an isomorphism
\begin{equation}\label{JRel} {\bf J} \cong H_2(X,E_p)\end{equation}
and
$J \mapsto [J]$ induces an
isomorphism \begin{equation}\label{JCycle}\{J \in {\bf J}| a(J) = 0\} \cong H_2(X)/\Z[E_p] \cong \{x\in H_2(X)\,|\, x\cdot \sigma(U) = 0\}\,.\end{equation} \end{thm}

\begin{remark} The hypothesis on $U$ can be weakened considerably at the cost of  added complications in the proof below. \end{remark}

\begin{proof} Recall from above that there exists a {topological} section $\sigma: U \to X$, which gives a splitting of the first map of the long exact homology sequence of a pair \cite{EilenbergSteenrod} :
\begin{equation} ....\to H_2(E_p) \to H_2(X) \to H_2(X,E_p) \to H_1(E_p) \to ...\end{equation}
since $[E_p] \cdot \sigma(U) = 1\,.$
The statement in (\ref{JRel}) follows then from (\ref{JCycle}) and the long exact sequence above.

The hypothesis on $U$ implies that there is a topological extension of $\pi$ to $\overline\pi:{\overline X} \to {\overline U}\subset \C\,,$ also a fiber bundle outside the points $\{q_1,...,q_N\}\,.$ There is a well defined topological intersection pairing \cite{Spanier}\cite{Bredon}\cite{Hatcher}

\begin{equation} H_2(X, E_p) \times H_2(\overline X - E_p, \overline X - X) \to \Z\,.\end{equation}
(Here is an intuitive geometric definition: Let $(A,\partial A)$ and $(B,\partial B)$ be oriented relative chains representing $\alpha$ and $\beta$ in these groups. Since $\partial A \cap \partial B = \emptyset\,,$ after an arbitrarily small deformation fixing the boundaries, it can be assumed $A$ and $B$ are in "general position", meaning that they intersect transversely in points in the interior of simplices. The intersection number $\alpha\cdot\beta$ will then be the number of these points, counted sign determined by the orientations. )

Possibly choosing $\epsilon_0$ above smaller, let $\hat\delta_j:[1-\epsilon_0] \to U$ be a small path, disjoint from $\delta_j$ except at the endpoint $\delta_j(1) = q_j = \hat\delta_j(1)\,,$ where the two paths meet  in one point. For example, in local coordinates as above around $q_j\,,$ take $\hat\delta_j(1-\epsilon) = \epsilon e^{2\pi i \hat\theta_j}$ for some angle $\hat\theta_j \ne \theta_j\,.$

Let $\hat C_{j,1-\epsilon}\,,$ $0 < \epsilon \le 1\,,$ be the corresponding vanishing cycle. Let

\begin{equation} \hat \Gamma_{j,\epsilon_0} = \bigcup_{0 <  \epsilon \le\epsilon_0}\hat C_{j,1-\epsilon} \cup \{\hat q_j\}\,.\end{equation}
be the corresponding local prong or thimble.
Let $\delta_{j,\infty}$ be a path from $\hat\delta_j(1-\epsilon_0) = \delta_{j,\infty}(1-\epsilon_0)$ to a point $\delta_{j,\infty}(0) \in \overline X - X\,,$
disjoint from the paths $\delta_k\,.$

\vskip 0.2in

%drawing -4$ here%
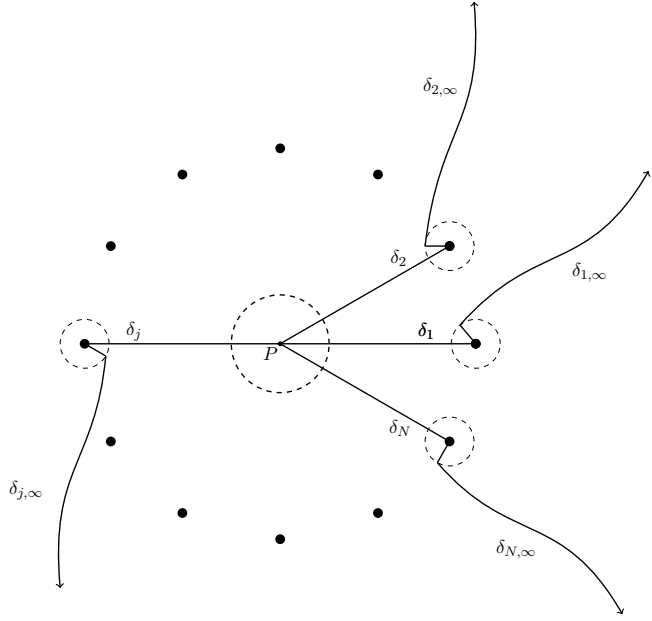
\begin{center}
\scalebox{.65}{
  \begin{tikzpicture}
    \fill[thick] (0:40mm) circle (1mm);
    %\node at (0:45mm) {$q_1$};
    \draw[thick] (0,0) -- (0:40mm);
    \draw[thick] (0:40mm) -- +(130:5mm);
    \draw[thick,->] (0:40mm)+(130:5mm) .. controls +(60:25mm) and +(240:25mm) .. +(45:50mm);
    \node at (5:30mm) {$\delta_1$};
    \fill[thick] (30:40mm) circle (1mm);
    \draw[dashed] (30:40mm) circle (5mm);
    %\node at (30:45mm) {$q_2$};
    \draw[thick] (0,0) -- (30:40mm);
    \draw[thick] (30:40mm) -- +(180:5mm);
    \draw[thick,->] (30:40mm)+(180:5mm) .. controls +(95:25mm) and +(-85:25mm) .. +(5mm,50mm);
    \node at (58:62mm) {$\delta_{2,\infty}$};
    \node at (36:30mm) {$\delta_2$};
    \fill[thick] (-30:40mm) circle (1mm);
    \draw[dashed] (-30:40mm) circle (5mm);
    %\node at (-30:45mm) {$q_N$};
    \draw[thick] (0,0) -- (-30:40mm);
    \draw[thick] (-30:40mm) -- +(240:5mm);
    \draw[thick,->] (-30:40mm)+(240:5mm) .. controls +(-60:25mm) and +(-240:25mm) .. +(-45:50mm);
    \node at (-35:30mm) {$\delta_N$};
    \fill[thick] (0:40mm) circle (1mm);
    \draw[dashed] (0:40mm) circle (5mm);
    %\node at (180:45mm) {$q_j$};
    \node at (5:30mm) {$\delta_1$};
    \draw[thick] (0,0) -- (180:40mm);
    \draw[thick] (180:40mm) -- +(-30:5mm);
    %\draw[thick] (180:40mm) arc (0:-40:4cm);
    \draw[thick,->] (180:40mm)+(-30:5mm) .. controls +(-85:25mm) and +(95:25mm) .. +(-5mm,-50mm);
    % \draw[thick](0,0) .. controls (1,1) .. (4,0)
    %  (5,0) .. controls (6,0) and (6,1) .. (5,2);
    % \draw[thick] (180:40mm)+((-3mm,00mm) arc (50:90:4cm) ;
    \node at (175:30mm) {$\delta_j$};
    \draw[dashed] (180:40mm) circle (5mm);
    \fill[thick] (30:40mm) circle (1mm);
    \fill[thick] (180:40mm) circle (1mm);
    \fill[thick] (60:40mm) circle (1mm);
    \fill[thick] (90:40mm) circle (1mm);
    \fill[thick] (120:40mm) circle (1mm);
    \fill[thick] (150:40mm) circle (1mm);
    \fill[thick] (210:40mm) circle (1mm);
    \fill[thick] (240:40mm) circle (1mm);
    \fill[thick] (270:40mm) circle (1mm);
    \fill[thick] (300:40mm) circle (1mm);
    \draw[thick,dashed] (0mm,0mm) circle (10mm);
    \node at (-2mm,-2mm) {$P$};
    \fill[thick] (0,0) circle (.5mm);
    %\draw[thick] (0:45mm) arc (-90:0:3cm);
    %\draw[thick] (30:45mm) arc (-60:30:3cm);
    %\draw[thick] (-30:45mm) arc (60:0:3cm);
    \node at (13:65mm) {$\delta_{1,\infty}$};
    \node at (-41:64mm) {$\delta_{N,\infty}$};
    \node at (210:60mm) {$\delta_{j,\infty}$};
\end{tikzpicture}
}
\end{center}

Let
\begin{equation} \Gamma_{j,\infty} = \Psi_{j,\infty}([0,1-\epsilon]\times
\hat C_{j,1-\epsilon_0})\cup \bigcup_{0 <  \epsilon \le\epsilon_0}\hat C_{j,1-\epsilon} \cup \{\hat q_j\} = \Psi_{j,\infty}([0,1-\epsilon]\times \hat C_{j,1-\epsilon_0})\cup\hat \Gamma_{j,\epsilon_0}\,,\end{equation}
be the corresponding prong, $\Psi_{j,\infty}$ a trivialization of $\overline\pi|{\overline\pi}\,^{-1}\delta_{j,\infty}([0,1-\epsilon_0])\,.$

Then

\begin{equation} \Gamma_{j,\infty}\cap \Gamma_j = \{q_j\}\end{equation}
transversely and

\begin{equation}
\Gamma_{j,\infty}\cap \Gamma_k = \emptyset \end{equation}
for $k \ne j\,.$ Therefore the classes $[ \Gamma_{j,\infty}] \in H_2(\overline X - E_p, \overline X - X)$ and $[\Gamma_k] \in H_2(X,E_p)$ satisfy

\begin{equation}\label{sign} [\Gamma_j]\cdot [ \Gamma_{j,\infty}] = \pm 1\qquad ; \qquad [\Gamma_k]\cdot [ \Gamma_{j,\infty}] = 0 \qquad{k\ne j}\,.\end{equation}
(The non-zero intersection number is actually $-1;$ this will be discussed in more detail in the proof of Theorem \ref{int}.)
Therefore the map $J \mapsto [J]_p$ induces an isomorphism of $\bf J$ onto a summand of $H_2(X,E_p)$ of rank $N\,.$

If $Z \to S^1$ is a smooth fiber bundle over a circle with fiber $F\,,$
then there is a monodromy $\phi: F \to F$ such that $Z$ is diffeomorphic to the mapping cylinder $F \times [0,1]/(x,0) \thicksim (\phi(x),1)$ via a diffeomorphism that carries the bundle projection to the map to map $(f,t) \mapsto e^{2\pi i t}\,.$ It follows from the Mayer Vietoris sequence \cite{EilenbergSteenrod}(or a spectral sequence argument \cite{McCleary}) that there is long exact sequence

\begin{equation}...\to H_k(F)^ {\textstyle{\phi_*-1\atop \rightarrow}} H_k(F) \to H_k(E) \to H_{k-1}(F ) \to...\end{equation}
Further, the fiber bundle has a section $\sigma: S^1 \to Z$ if and only if $\phi$ has a fixed point.

We apply this to $\pi^{-1}B_j(\epsilon_0)\,,$ the ball of radius $\epsilon_0$ in the local coordinates around $q_j\,.$ The boundary of this manifold is a bundle over the circle $\partial B_j(\epsilon_0)\,,$ with torus fiber. The monodromy is the the map $\psi_{i,1-\epsilon}\,.$ It follows that $H_1(\pi^{-1}\partial B_j(\epsilon_0)) \cong \Z\oplus \Z$ generated by the orientation class of a fiber and $[C_{j,1-\epsilon_0}\times \sigma(\partial B_j(\epsilon_0))]\,,$ and $H_2(\pi^{-1}\partial B_j(\epsilon_0)) \cong \Z\oplus \Z$ generated by any  class $\tau$ with $\tau \cdot \gamma_j = 1\,.$ and the class $[\sigma(\partial B_j(\epsilon_0))\,.$ However $\pi^{-1}B_j(\epsilon_0)$ collapses homotopically equivalently to the singular fiber over $E_{q_j}\,,$ whose second homology is generated by the image of the orientation class under the collapse, and whose first homology by the image of  $\tau\,,$ i.e. the collapse of a class represented by a curve $D_1$ that meets $C_{j,1-\epsilon}$ in one point. Also, $\sigma(\partial B_j(\epsilon_0)) = \partial \sigma(B_j(\epsilon_0))\,.$  Hence $H_2( \pi^{-1}B_j(\epsilon_0),\pi^{-1}\partial B_j(\epsilon_0)) \cong \Z\,,$ generated by the homology class of
$\sigma(B_j(\epsilon_0),\partial B_j(\epsilon_0))\,.$

Let $X_0 = \pi^{-1}U_0\,,$ $U_0 = U - \bigcup_1^N B_j(\epsilon_0)^\circ\,.$ We now claim the map

\begin{equation} H_2(X,E_p) \to H_2(X,X_0) \cong \bigoplus_{1}^N  H_2( \pi^{-1}B_j(\epsilon_0),\pi^{-1}\partial B_j(\epsilon_0))\,\end{equation}
is trivial. In fact, the image of the connecting homomorphism in the long exact homology sequence of the triple $(X,X_0,E_p) $ is generated by the classes
$[\sigma(\partial B_j(\epsilon_0))]\in H_1(X_0,E_p)\,,$
$\pi_*[\sigma(\partial B_j(\epsilon_0))] = [\partial B_j(\epsilon_0))] \in H_1(U_0,p)\,,$ and for $j=1,...,N$ these classes form a basis of this group. Therefore the connecting homomorphism is one-to-one.

More precisely, the inclusion
\begin{equation}\bigcup_1^N\delta_j([0,1-\epsilon_0]) \cup B_j(\epsilon_0)) \hookrightarrow U_0\end{equation}
is a homotopy equivalence, and hence so is

\begin{equation}\pi^{-1}\bigcup_1^N\delta_j([0,1-\epsilon_0]) \cup B_j(\epsilon_0)) \hookrightarrow X_0\end{equation}
It follows by excision that

\begin{equation} H_2(X_0,E_p) \cong H_2(X_0,\pi^{-1}\bigcup_1^N \delta_j([0,1-\epsilon_0]) \cong \bigoplus_1^N H_2(\pi^{-1}\partial B_j(\epsilon_0)),
E_{\delta_j(1-\epsilon_0)})\,.\end{equation}
Also by excision, for any bundle $Z \to S^1$ as discussed above, the quotient map induces and isomorphsim $H_i(F \times [0,1], F \times \{0,1\}) \to H_i(Z,F)\,.$ (This is part of one proof of the above long exact sequence.) Therefore, in our case
$H_2(\pi^{-1}\partial B_j(\epsilon_0)),
E_{\delta_j(1-\epsilon_0)}) \cong \Z\oplus\Z$ and one of the generators is represented by the closed class (i.e. in the image of the absolute homology) represented by the torus $[C_{j,1-\epsilon_0}\times \sigma(\partial B_j(\epsilon_0))]\,.$ This class obviously maps to zero in $H_2(X, E_p)\,;$ therefore this group is a quotient of a free abelian group of rank $N\,.$ Therefore, $J \mapsto [J]_p\,,$ already shown to be a one-to-one map onto a summand, is an isomorphism.
\end{proof}

We now outline how our construction also provides some topological information about the representation of homology classes in $H_2(X)$ by embedded spheres:

\begin{thm}\label{emb}(See also  \cite{GrassiHalversonShaneson2013}) Let $J = (J_1,...,J_N) \in \bf J$ with $a(J) = 0$
and $|J_i| \le 2\,.$ Then $[J] \in H_2(X)$ is represented  by a smoothly embedded
2-sphere $S^2 \subset X\,.$

\end{thm}
\begin{proof}(Outline) As  mentioned above, the construction of prongs and the proof of \ref{jct} can be used to provide an explicit geometric cycle representing $[J]\,,$ $J$  a junction with $a(J) = 0\,.$ Suppose first $|J_i| \le 1\,, $ $1 \le i \le N\,.$ The cycle representing $[J]$ can be then described as a union of some of the (oriented) prongs $\pm \Gamma_{j,t}$ for a small value of $t\,,$ together with a punctured sphere (a sphere with some disks removed), in a product neighborhood $E_0 \times D^2$ of the smooth fiber, that bounds the union of the bounding curves $\pm C_{j,t}\,.$ The construction of the punctured disk is indicated \cite{GrassiHalversonShaneson2013}, and, as explained there, the union of the prongs and the punctured disk is a smoothly embedded $S^2\,.$
This proves the case with $|J_i| \le 1$.
 For $|J_j| = \pm 2$ we need to consider $\Gamma_{j,t}$ and a parallel copy $\hat\Gamma_{j,t}$ of the corresponding prong, using a slight deformation of the path $\delta_j$ in the base. This construction is described in more detail below in the proof of the next theorem. The prong and the deformed prong can then be added to one another in a way that eliminates the intersection point at the singularity of $E_j$ and provides an annulus bounding  $\pm C_{j,t} \cup \pm\hat C_{j,t}\,.$ These annuli, the oriented prongs with coefficients $\pm 1\,,$ and the punctured sphere  bounding the union of the vanishing cycles in the different nearby smooth fibers again provide a smoothly embedded $S^2$ representing $[J] \in H_2(X)\,.$
\end{proof}
This result is similar in spirit to the (simpler) topological fact that in $H_2(\PP^2)$ a generator or twice a generator can be represented by a smoothly embedded $S^2\,$\cite{Hartshorne};  however, in $X$ these representatives are not algebraic. As is the case with $\PP^2$ \cite{Fox-Milnor}, it appears that every element of $H_2(X)$ can be represented by an topologically embedded $S^2$ with a non-locally smoothable point. It would be interesting the determine the minimal genus of a smoothly embedded closed (oriented) real 2-manifold representing a given general element of $H_2(X)\,.$

\begin{defin}\label{intj}Consider a junction $J$. Define a self-intersection
  \begin{equation}\label{langle}
    \langle J, J \rangle = -\sum_{k > j \ge 2}J_kJ_j\, \gamma_k\cdot\gamma_j - \sum_{j=1}^N J_j^2
  \end{equation}
  \end{defin}

\begin{thm}\label{int}
  Let  $J$be a junction with
 $a(J) = 0\,. $ Then the topological intersection number satisfies
  \begin{equation} [J] \cdot [J] = \langle J, J \rangle\,.\end{equation}
\end{thm}
\begin{remark} Different formulae are obtained by cyclically permuting the indices, i.e. rotating the small circle around $p\,.$ All these satisfy the conclusion of the Theorem. In fact, let $\hat p$ be a point near $p$ that lies between the paths $\delta_1$ and $\delta_2\,.$ Then there will be defined  a small ``deformation" of the paths $\delta_j\,,$ $j=1,...,N\,,$ to paths $\hat \delta_j$ from $\hat p$ to the points $q_j\,,$ which will be described in the proof. The corresponding prongs will actually represent classes
$[J]_p \in H_2(X- E_{\hat p},E_p)$ and $[J]_{\hat p} \in H_2(X - E_p,E_{\hat p})\,.$  There is also a well defined topological intersection pairing on these groups, and
it will be shown that

\begin{equation}[J]_p \cdot [J]_{\hat p} = \langle J, J \rangle\,.\end{equation}
Thus the above formula for classes which do not have asymptotic charge zero has a topological interpretation in terms of a deformation of the prongs, and all the possible formulas correspond to the different possible small deformations, up to a suitable notion of homotopy. 
\end{remark}
\begin{remark}
The missing index $j =1$ in the first term on the right side of \ref{langle} corresponds to the fact that in the deformation used below (see the figure in the proof), since the deformed basepoint $\hat p$ lies between between $\delta_1$ and $\delta_2\,,$ we can reach $q_1$ with a path near $\delta_1$ that does not intersect $\delta_1$ except at $q_1\,.$ We could also reach $q_2$ without crossing $\delta_2\,$ but on the opposite side of $\delta_2$ from that indicated in the figure, leading to a  change in the formula in the corresponding place in the second term.
\end{remark}
\begin{remark} An intersection pairing is defined by
\begin{equation} \langle J, K \rangle = {1\over 2}(\langle J+K, J + K\rangle - \langle J, J \rangle - \langle K, K \rangle)\,.\end{equation} It follows that

\begin{equation} [J] \cdot [K] = \langle J, K \rangle\,.\end{equation}
if $a(J) = a(K) = 0\,.$\end{remark}

\begin{proof}(of Theorem \ref{int}) We start by determining the sign in equation (\ref{sign}). We keep the same notation. It was shown in the proof of Theorem \ref{jct} that $H_2(\pi^{-1}B_j(\epsilon_0), E_{\delta_j(1-\epsilon_0)})$ is infinite cyclic and the connecting homomorphism to
$H_1(E_{\delta_j(1-\epsilon_0})$ is injective (with image generated by the class of the local vanishing cycle.) In the general fiber  bundle $Z \to S^1$ with monodromy $\phi$ as in the previous proof, let $w$ be a $k-1$ cycle of $F\,.$ Then the image of the relative $F \times\{0,1\}$  cycle $w \times [0,1]$ in the quotient $Z$ represents an element $S(w) \in H_k(Z,F)$ with
$\partial S(w) = (\phi_*[w]-[w]) \in H_1(F)\,.$ In our case, take for $w$ a curve with

$[w] \cdot [C_{j,1-\epsilon_0}] = -1\,.$ Then by \ref{pl}, $\partial S(w) = [C_{j,1-\epsilon_0}] \in H_1(E_{\delta_j(1-\epsilon_0}))\,.$ Let

\begin{equation} i_*:H_2(\pi^{-1}\partial B_j(\epsilon_0),E_{\delta_j(1-\epsilon_0)}) \to H_2(\pi^{-1}B_j(\epsilon_0), E_{\delta_j(1-\epsilon_0)})\end{equation}
be the map induced by inclusion. Then $\partial i_*S(w) = [C_{j,1-\epsilon_0}]$ also, by naturality. Therefore $i_*S(w) = [\Gamma_{j,\epsilon_0}]\,, $ the class of the local ``prong" corresponding to the restriction of $\delta_j$ to $[1-\epsilon_0,1]\,.$
The intersection pairing is also defined between elements of $H_2(X,E_{\delta_j(1-\epsilon_0)})\,,$ to which the target of $i_*$ maps by another inclusion induced map, and elements of $H_1(\overline X - E_{\delta_j(1-\epsilon_0)},\overline X - X)\,,$ in which $\Gamma_{j,\infty}$ represents an element. Therefore

\begin{equation}[\Gamma_{j,\epsilon_0}]\cdot [ \Gamma_{j,\infty}] = S(w)\cdot [\hat \Gamma_{j,\infty}] = [w \times \{\hat\theta_j\}]\cdot [\hat C_{j,1-\epsilon_0}] =[w]\cdot [C_{j,1-\epsilon_0}] = -1
\end{equation}
It follows from the local nature of intersection numbers that the local prongs $\Gamma_{j,\epsilon_0}$ and $\hat \Gamma_{j,\epsilon_0} $ meet with intersection number -1.

Let $\hat p$ be near $p\,,$ lying between the paths $\delta_1$ and $\delta_2\,,;$ for example, take $\hat p =
 \hat re^{2\pi \eta_1}$ in local the local coordinates near $p\,,$ $\hat r < r_0$ and $0 < \eta_1 < N^{-1}$ both small; we assume for $0 \le r \le r_0\,,$ $\delta_j(r) = re^{2\pi i\over N}$ in the local coordinates around $p\,,$ as above. Define paths $\hat\delta_j$ from $\hat p$ to $q_j$ in four parts as follows: First take a straight line from $\hat q$ to $r_j\,e^{2\pi i\eta}\,,$ $\hat r < r_1 < r_2 < .... < r_j < r_0\,.$ Follow this with a circular arc $r_je^{2\pi i t}\,,$ $\eta \le t \le {j-1\over N} + \eta\,.$ Parameterize what has been done so far so that $\hat \delta_j(0) = \hat p$ and $\hat\delta_j(r_j) =  r_je^{2\pi i({j-1\over N} +\eta_j)}\,.$ Then follow with paths parallel to $\delta_j\,,$ until we reach $B_j(\epsilon_0)$ at a point $\epsilon_0e^{2\pi i\hat\theta_j}$ in local coordinates around $q_j\,,$ and then follow the straight line to $q_j\,.$
 %drawing -2 here$

\vskip 0.2in

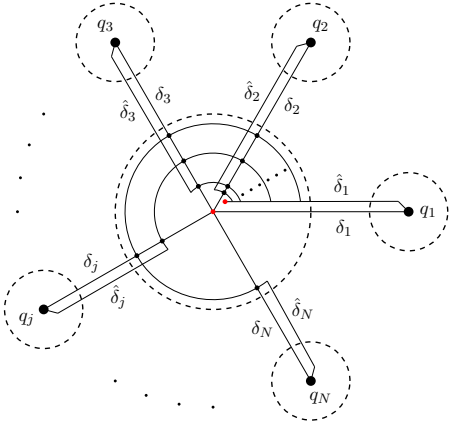
\begin{center}
\scalebox{.65}{
  \begin{tikzpicture}
    % FILLED EXTERIOR CIRCLES
    \fill[thick] (0:40mm) circle (1mm);
    \fill[thick] (60:40mm) circle (1mm);
    \fill[thick] (120:40mm) circle (1mm);
    \fill[thick] (150:40mm) circle (.3mm);
    \fill[thick] (160:40mm) circle (.3mm);
    \fill[thick] (170:40mm) circle (.3mm);
    \fill[thick] (180:40mm) circle (.3mm);
    \fill[thick] (210:40mm) circle (1mm);
    \fill[thick] (240:40mm) circle (.3mm);
    \fill[thick] (250:40mm) circle (.3mm);
    \fill[thick] (260:40mm) circle (.3mm);
    \fill[thick] (270:40mm) circle (.3mm);
    \fill[thick] (300:40mm) circle (1mm);
   % DASHED EXTERIOR NEIGHBORHOODS
    \draw[thick,dashed] (0:40mm) circle (8mm);
    \draw[thick,dashed] (60:40mm) circle (8mm);
    \draw[thick,dashed] (120:40mm) circle (8mm);
    \draw[thick,dashed] (210:40mm) circle (8mm);
    \draw[thick,dashed] (300:40mm) circle (8mm);
    % DRAW PATHS IN ORDER
    % PATH ONE
    \draw (0:0mm) -- (0:40mm);
    \draw (0+0:40mm)+(135+0:3mm)--(0+0:40mm);
    \draw (0+0:40mm)+(135+0:3mm)--(40:3.2mm);
    %\draw (0+0:40mm)+(-135+0:2mm)--(0+0:40mm);
    % PATH TWO
    \draw (0+60:0mm) -- (0+60:40mm);
    \draw (0+60:40mm)+(135+60:3mm)--(0+60:40mm);
    \draw (0+60:40mm)+(135+60:3mm)--(0+87:4.5mm);
    \draw (87:4.5mm) arc (87:28:4.5mm);
    % PATH THREE
    \draw (0+120:0mm)--(0+120:40mm);
    \draw (0+120:40mm)+(135+120:3mm)--(0+120:40mm);
    \draw (0+120:40mm)+(135+120:3mm)--(0+139:6mm);
    \draw (0+139:6mm) arc (139:20:6mm);
    % PATH FOUR
    \draw (0+210:0mm)--(0+210:40mm);
    \draw (0+210:40mm)+(135+210:3mm)--(0+210:40mm);
    \draw (0+210:40mm)+(135+210:3mm)--(220:12mm);
    \draw (220:12mm) arc (220:10:12mm);
    % PATH FIVE
    \draw (0+300:0mm)--(0+300:40mm);
    \draw (0+300:40mm)+(135+300:3mm)--(0+300:40mm);
    \draw (0+300:40mm)+(135+300:3mm)--(308:18mm);
    \draw (308:18mm) arc (308:6.5:18mm);
    % DRAW SKIPPING DOTS
    \fill (30:7mm) circle (.3mm);
    \fill (30:9mm) circle (.3mm);
    \fill (30:11mm) circle (.3mm);
    \fill (30:13mm) circle (.3mm);
    \fill (30:15mm) circle (.3mm);
    \fill (30:17mm) circle (.3mm);
    % PATH DOTS
    \fill (180+30:12mm) circle (.5mm);
    \fill (180+30:18mm) circle (.5mm);
    \fill (360-60:18mm) circle (.5mm);
    \fill (90+30:12mm) circle (.5mm);
    \fill (90+30:6mm) circle (.5mm);
    \fill (90+30:18mm) circle (.5mm);
    \fill (90-30:12mm) circle (.5mm);
    \fill[red] (40:3.2mm) circle (.5mm);
    \fill (60:4.5mm) circle (.5mm);
    \fill (60:6mm) circle (.5mm);
    \fill (60:18mm) circle (.5mm);
    % Q LABELS
    \node at (0:44mm) {$q_1$};
    \node at (60:44mm) {$q_2$};
    \node at (120:44mm) {$q_3$};
    \node at (180+30:44mm) {$q_j$};
    \node at (300:44mm) {$q_N$};
    % DELTA LABELS
    \node at (-7+0:27mm) {$\delta_1$};
    \node at (10+0+2:27mm) {$\hat\delta_1$};
    \node at (-7+60:27mm) {$\delta_2$};
    \node at (10+60+2:26mm) {$\hat\delta_2$};
    \node at (-7+120:26mm) {$\delta_3$};
    \node at (10+120:27mm) {$\hat\delta_3$};
    \node at (-7+210:27mm) {$\delta_j$};
    \node at (10+210+2:26mm) {$\hat\delta_j$};
    \node at (-7+300:26mm) {$\delta_N$};
    \node at (10+300+3:27mm) {$\hat\delta_N$};
    % CENTER CIRCLE AND P
    \draw[thick,dashed] (0mm,0mm) circle (20mm);
    %\node at (-4mm,2mm) {$P$};
    \fill[thick,red] (0,0) circle (.5mm);
    \end{tikzpicture}
    }
\end{center}

 Let $\hat \Gamma_j$ be the corresponding prongs. For $r \le r_0\,,$

 \begin{equation}i_r:H_2(X,E_{\delta_j(r)}) \to H_2(X,\pi^{-1}B_r(p)) \qquad \hat i_{r}:H_2(X,E_{\hat\delta_j(r)}) \to H_2(X,\pi^{-1}B_r(p)) \end{equation}
 be the maps induced by inclusion; these are isomorphisms (as the fiber bundle is trivial over $B_r(p)\,.$ Then it follows by letting $\eta,\hat r \to 0$ that
 \begin{equation} \hat i_r[\hat \Gamma_{j,r}] = i_r[\Gamma_{j,r}]\end{equation}
 and that  under the inclusion induced isomorphisms, $H_1(E_{\delta_j(r)})\cong H_1(\pi^{-1}B_r(p))$
 and $H_1(E_{\hat\delta_j(r)}) \cong H_1(\pi^{-1}B_r(p))\,,$ the vanishing cycles $[C_{j,r}]$ and $ [\hat C_{j,r}]$  have the same image. The orientation classes of the fibers also have the same image in $H_2\,.$ In particular if $\hat\delta_k(s) = \delta_j(t)\,,$ then
 \begin{equation}[\hat C_{k,s}]\cdot [ C_{j,t}] = [C_k]\cdot [ C_j] = [\hat C_k]\cdot [ \hat C_j]\,.\end{equation}

 Now we compute the intersection numbers of the two sets of prongs. If $k < j\,,$ then $\hat \Gamma_k \cap \Gamma_j = \emptyset\,;$ hence $[\hat \Gamma_k] \cdot [\Gamma_j] = 0\,.$
 \begin{equation} \hat \Gamma_1 \cap \Gamma_1 = \{q_1\}\,.\end{equation}
 It follows from the first paragraph that  $[\hat \Gamma_1] \cdot [\Gamma_1] = -1\,.$
 For $2 \le j \le N\,,$
 \begin{equation}
 \hat \Gamma_j \cap \Gamma_j = \{q_j\} \cup \Big[\hat C_{j,\hat\delta_j(s_j)}\cap  C_{j,\hat\delta_j(r_j)}\Big]\,,\end{equation}
 $s_j $ the unique value with $\hat \delta_j(s_j) = \delta_j(r_j) = r_je^{2\pi i(j-1)\over N}\,.$ At the point of intersection the path $\hat\delta$ meets $\delta$ with sign $-1\,;$ hence
 \begin{equation} [\hat \Gamma_j] \cdot [\Gamma_j] = -1 - [\hat C_{j,s_j}]\cdot [C_{j,r_j}] = -1 - [C_j]\cdot [C_j] = -1\,.\end{equation}
 Finally, if $k > j\,,$
 \begin{equation}
 \hat \Gamma_k \cap \Gamma_j = \hat C_{k,\hat\delta_j(s_{k,j})}\cap  C_{j,\hat\delta_j(r_j)}\,,\end{equation}
 $s_{k,j} $ the unique value with $\hat \delta_j(s_{k,j}) = \delta_j(r_j) = r_je^{2\pi i(j-1)\over N}\,.$
 Therefore
 \begin{equation} [\hat \Gamma_k] \cdot [\Gamma_j] =  [\hat C_{k,s_{k,j}}]\cdot [C_{j,r_j}] = [C_k]\cdot [C_j] \end{equation}

 Hence by bilinearity,

\begin{equation}
\Big(\sum_1^N J_j[\hat \Gamma_j]\Big)\cdot \Big(\sum_1^N J_j[ \Gamma_j]\Big) = -\sum_{k > j \ge 2}J_kJ_j\, \gamma_k\cdot\gamma_j - \sum_{j=1}^N J_j^2\,.\end{equation} In other words, in the junction notation, if $J = (J_1,...,J_N) $
\begin{equation} [J]_{\hat p} \cdot [J]_p =  \langle J, J \rangle\,.\end{equation}
Clearly, from the preceding paragraph, $\partial [J]_p = 0$ if and only if $\partial [J]_{\hat p} = 0\, .$ In this  case it is clear that any two closed classes $A,B \in H_2(X)$ with images
$[J]_p$ and $[J]_{\hat p}$ satisfy $B\cdot A = \langle J, J \rangle$ and have the same image in $H_2(X,\pi^{-1}B_r(p)) \cong H_2(X,E_p)\,.$ Therefore $A$ and $B$ are equal up to a multiple of the orientation class of $E_p\,;$ imposing the condition
$A\cdot \sigma(U) =  B\cdot\sigma(U) = 0$ then implies $A = B = [J]\,,$ so
$[J]\cdot [J] = \langle J, J \rangle\,.$
\end{proof}

\section{Deformations, String Junctions,  Lie Algebras and more}\label{DefAndMore}
In this section we consider deformations of elliptic surfaces and the
appearance of string junctions in the deformed geometry.  Let
$\pi:X\longrightarrow U$ be an elliptic fibration over an open set
$U\subset \C$ with section $\sigma$, with $\pi_W: W\longrightarrow U$
its associated Weierstrass model; suppose that the ramification divisor $\Sigma $ consists of the origin, namely $\Sigma = \{0\} \subset
U$.  $W$ has local equation $ y^2=x^3+f x+g$, then the possible singular fibers are described in the following table (see \cite{Mi81}):

\begin{center}
\begin{tabular}{ccccc}
  \shs\shs Kodaira Fiber Type \shs\shs & \shs{\it ord(f)}\shs & \shs{\it ord (g)}\shs & \shs{\it ord($\Delta$)}\shs   &
   \shs Singularity Type \shs \\
  \hline \hline
  smooth & $\geq 0$ & $\geq 0$ & $0$ &  $-$\\
  $I_n$  & $0$ & $0$ & $n$ &   $\T_{n-1}$ \\
  $II$ & $\geq 1$ & $1$ & $2$& $-$\\
  $III$ & $1$ & $\geq 2$ & $3$ &  $\T_1$ \\
  $IV$  & $\geq 2$ & $2$ & $4$ &   $\T_2$\\
  $I_n ^*$ & $2$ & $\geq 3$ & $n+6$ &  $\D_{n+4}$ \\
  $I_n ^*$ & $\geq 2$ & $3$ & $n+6$ &  $\D_{n+4}$ \\
  $IV^*$ & $\geq 3$ & $4$ & $8$ &  $\E_6$ \\
  $III^*$  & $3$ & $\geq 5$ & $9$ &  $\E_7$\\
  $II^*$  & $\geq 4$ & $5$ & $10$ &  $\E_8$ \\ \hline \hline
\end{tabular}
\end{center}

\vskip 0.1in

\begin{thm}
  \begin{itemize}
  \item[(i)] There exists a deformation $W_0$  of the Weierstrass equation, such that
    $\pi_0:W_0\rightarrow U$ is an elliptic fibration with $\pi_0^{-1}(q)$
    an $I_1$ singular fiber $\forall q \in \Sigma$.
  \item[(ii)] Let $J^{(-2)}$ be the set of junctions $J$ of $W_0$ such
    that $a(J)=0$ and $\langle J, J \rangle = -2$. Equivalently, let $J^{(-2)} \subset H_2(W_0)$ consists of those elements $x$ with $x\cdot x = -2$ and $x\cdot\sigma(U) = 0\,.$ Then $\sharp\,J^{(-2)}$ is
    the number of roots of the Lie algebra associated to the singularity of $\pi$.
  \item[(iii)] If $J = (J_1,...,J_N) \in J^{(-2)}\,,$ then $|J_i| \le 1\,.$ In particular, all elements of $J^{(-2)}$ are represented by smoothly embedded $S^2$'s.
  \item[(iv)] There exist  subsets
    $\{\alpha_1,...,\alpha_r\}\subset J^{(-2)}$ with $r$ elements such that $\langle \alpha_i, \alpha_j \rangle$
    is the negative Cartan matrix associated to the singularity of $\pi$.
  \end{itemize}
\end{thm}

\begin{proof}
We will illustrate the proof in three of the cases from the table. All the singularities on the table can be handled in the same way, see \cite{GrassiHalversonShaneson2014, GrassiHalversonShaneson2013}.
For the first case,
assume that $\pi^{-1}(0)$ is of type $I_{r+1}\,,$ in other words, an $\T_r$ singularity. Then by \cite{Mi81} the defining  equation in the complement of the image of the section $\sigma$ can be written near $\pi^{-1}(0)$ as
  \begin{equation}
    y^2 = x^3 - 3a^2 x + 2a^3 + s^{r+1}
  \end{equation}
The derivative on the right hand side vanishes for $x = \pm \,a\,.$ Therefore, assuming $U$ was small enough to exclude the $r+1$-st roots of $-2a^3\,,$ the singular locus is  $\Sigma= \{0\}$.

In this case, take for  $W_0$ be the deformation of
$W$ defined by
\begin{equation}
    y^2 = x^3 - 3 a^2 x + 2 a^3 + s^{r+1} + \epsilon
\end{equation}
for $\epsilon \in \C$. For $|\epsilon|$ small enough, the new discriminant locus of this equation will intersect $U$ in the singular set of the fibration
\begin{equation}\Sigma_\epsilon = \{e^{2\pi ij\over {r+1}}\epsilon_0\,|\,j = 0,...,r\}\end{equation}
with $\epsilon_0^{r+1} = \epsilon$ a specific $r+1$-st root. The fiber $\pi^{-1}(s) - \sigma(s)$ is the two fold branched cover of $\C$ branched along the roots of
$x^3 - 3 a^2 x + 2 a^3 + s^{r+1} + \epsilon = 0\,;$ in particular at each point of $\Sigma_\epsilon$ there is a multiple root corresponding to an $I_1$-singularity. Let $\delta_j(t) = te^{2\pi ij\over {r+1}}\epsilon_0\,,$ $0\le t\le  1\,, $be the straight line path from the origin to the $j$th point in $\Sigma_\epsilon\,.$ Then the equation of $\pi^{-1}(\delta_j(t)) - \sigma(s)\,,$
\begin{equation}
    y^2 = x^3 - 3 a^2 x + 2 a^3 + t^{r+1}\epsilon_0 + \epsilon\,,
\end{equation}
is independent of $j\,.$ For example, if $\epsilon$ is real and positive and we also take $\epsilon_0$ to be real and positive, then as we move along each path from zero to one, the two imaginary roots converge to a common real real of multiplicity two at the end point and the real root remains always real. In any case, whether we set it up this way or not, it follows from the preceding section that the vanishing cycles $\gamma_j$ are all equal; $\gamma_1 = ....= \gamma_{r+1}\,.$ From this it then also follows that
$\{J \in \bold J| a(J) = 0\}$ has the basis $\alpha_1,...,\alpha_r\,,$ with $\alpha_1 = (1,-1,0,...,0)\,,$ $\alpha_2 = (0,1,-1,0,...,0)\,,$ etc. Since $\gamma_j \cdot\gamma_k = 0$ because all these classes are equal, $\langle \alpha_j,\alpha_j\rangle = -2\,,$ $\langle \alpha_j,\alpha_k\rangle = 1$ for $|j- k| = 1$ and zero for $|j-k| \ge 2\,.$ This clearly implies the result for this case, we get the roots of the Dynkin diagram and Cartan matrix of $\su(r)\,.$

Next consider the case of singular fiber  $\pi^{-1}(0)$ of type $III\,,$ defined by
  \begin{equation}
    y^2 = x^3 + sx + s^2\,.
  \end{equation}
In this case for $U$ small enough the the discriminant locus is only the origin. A deformation  can be  defined by
  \begin{equation}
    y^2 = x^3 + (s+\epsilon)x + (s^2+\epsilon)
  \end{equation}
Then it is not hard to see that, for $U$ small enough,
$\Sigma_\epsilon = \{q_1,q_2,q_3 \}$ consists of three points. We take the basepoint to the origin $s = 0\,,$ which is now a smooth fiber $E_0 = \pi_\epsilon^{-1}(0)$ of the deformed fibration $\pi_\epsilon\,.$ The fiber  $E_0 - \sigma(0)$ is the double branched cover of $\C$ along the roots of $x^3 + \epsilon x + \epsilon = 0\,.$ The inverse image in the branched over of the three lines joining  these roots determine curves representing elements $Z_i \in H_1(E_0)\,;$ choose the orientations so that $Z_1 + Z_2 + Z_3 = 0\,.$ In this case the algorithm in the previous section yields, for the three vanishing cycles, $\gamma_i = Z_i\,.$ In \cite{GrassiHalversonShaneson2014} this is obtained for a specific choice of small $\epsilon$ using a computer program solving cubics, but the result can also be obtained (tediously) by hand. Therefore there are exactly two junctions, $J = (1,1,1)$ and its negative, with $\langle J , J \rangle = -2$ and $a(J) = 0\,,$ and we obtain the roots and Cartan matrix of  $\su (2)\,.$

The third case, of type IV, is given by the equation

\begin{equation}
    y^2 = x^3 + s^2x + s^2\,.
  \end{equation}

 \begin{equation}
    y^2 = x^3 + (s^2 + 2\epsilon)x + s^2 + \epsilon\,.
  \end{equation}
provides a deformation  of the local model.
Near the origin there are now four points in the deformed discriminant. Using the same notation of the preceding paragraph for the homology classes determined by the roots, this time we will get the vanishing cycles $\gamma_1 = \gamma_3 = Z_1$ and $\gamma_2=\gamma_4 = Z_3$ for the set of ordered vanishing cycles, with the signs chosen so that $Z_1\cdot Z_ 3 = 1\,.$ Again, this is done with a computer program in \cite{GrassiHalversonShaneson2014} for a specific choice of $\epsilon\,,$ but it can also be worked out by hand. The set of junctions with $a(J) = 0$ therefore has dimension two with basis, for example,
${J_1,J_2} = \{(1,0,-1,0),(-,1,0,-1)\}\,,$ $\langle J_i,J_i\rangle = - 2\,,$ $\langle J_1,J_2\rangle = 1\,,$ there are six elements in $J^{(-2)}\,,$ and we get the roots and negative Cartan matrix of $\su(3)\,.$

\end{proof}
The example of the fiber of type IV actually arises from restricting a general Weierstrass model
$\pi:W_g\rightarrow \mathbb{F}_3$
for an elliptic Calabi-Yau threefold over the Hirzebruch surface $\mathbb{F}_3$.  This is an example of a ``non-Higgsable cluster'' (in physics language) with a type $IV$ fiber; for this fibration, there  there exists no smoothing deformation of the
  global model to a fibration with only $I_1$ singularities \cite{MorrisonTaylor2012NHC}. Nevertheless

 \begin{equation}
    y^2 = x^3 + (s^2 + 2\epsilon)x + s^2 + \epsilon\,.
  \end{equation}
provides a deformation  of the local model.

In the papers \cite{GrassiHalversonShaneson2014, GrassiHalversonShaneson2013} we actually derive the entire representation structure of the Lie algebra associated to the singularity using geometric string junctions. However, perhaps the main advantage of this method is its usefulness in considering higher dimensional elliptic fibrations (see also \cite{GrassiHalversonShanesonTaylor2014}).

\section{Higher Dimension, higher codimension}\label{G2}
 In physics, matter can appear when there is a codimension two stratum in the discriminant locus, arising as the intersection of two codimension one strata. Resolutions may be hard to use or may not even be available in all cases; we will conclude this paper with two illustrative examples of the deformation technique.

The example will be an elliptic fibration  with section over an open set in $\C^2$  whose discriminant locus is the union of two curves meeting transversely in  a point. Over each general point in the complement of the intersection, on one component  we have an $I_0^*$ singularity, on the other an $I_1$ singularity. Then without loss of generality it can be assumed there is a local equation in Weierstrass form \cite{Mi81}:

\begin{equation}y^2 = x^3 - 3c^2s^2x + 2c^3s^3 + ats^3\,.\end{equation}
Here $s,t$ are coordinates in the base. The fibration over the line obtained by fixing a non-zero value of $t$ has an $I_0^*$ singularity at $s = 0\,.$ For fixed $s \ne 0$ the fibration has an $I_1$ singularity. The deformation we consider varies with $t$:
\begin{equation}y^2 = x^3 - 3c^2s^2x + 2c^3s^3 + ats^3 + t\epsilon.\end{equation}
For a fixed $t \neq 0 $ and fixed $s$ there are two possible singular points, $y = 0, x = \pm cs\,.$ For $x = cs\,,$ the corresponding points on the singular locus are the three cube roots of $-{\epsilon\over a}\,$ denoted by the red dots in the picture below; for $x = -cs$ the points are the cube roots of $-{\epsilon t\over at + 4c^3}\,.$
denoted by the blue dots in the picture below.

\vskip 0.2in

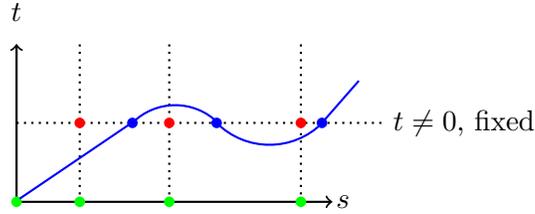
\begin{center}\label{G2disc}
%\scalebox{.65}{
  \begin{tikzpicture}[scale=.7]
    \draw[thick, ->] (0cm,0cm) -- (6cm,0cm);
    \draw[thick, dotted] (0cm,1.5cm) -- (7cm,1.5cm);
    \draw[thick, ->] (0cm,0cm) -- (0cm,3cm);
    \node at (6.2cm,0cm) {$s$};
    \node at (0cm,3.6cm) {$t$};
    \node at (8.5cm,1.5cm) {$t\ne0$, fixed};

    \draw[dotted,thick] (1.2cm,0cm) -- (1.2cm, 3cm);
    \draw[dotted,thick] (2.9cm,0cm) -- (2.9cm, 3cm);
    \draw[dotted,thick] (5.4cm,0cm) -- (5.4cm, 3cm);

    \fill[color=green] (0cm,0cm) circle (1mm);

    \fill[color=green] (1.2cm,0cm) circle (1mm);
    \fill[color=green] (2.9cm,0cm) circle (1mm);
    \fill[color=green] (5.4cm,0cm) circle (1mm);

    \fill[color=red] (1.2cm,1.5cm) circle (1mm);
    \fill[color=red] (2.9cm,1.5cm) circle (1mm);
    \fill[color=red] (5.4cm,1.5cm) circle (1mm);

    \fill[color=blue] (2.2cm,1.5cm) circle (1mm);
    \fill[color=blue] (3.8cm,1.5cm) circle (1mm);
    \fill[color=blue] (5.8cm,1.5cm) circle (1mm);

    \draw[color=blue,thick] (.06cm,.06cm) -- (2.2cm,1.51cm);
    \draw[color=blue,thick] (2.2cm,1.5cm) arc (135:45:1.15cm);
    \draw[color=blue,thick] (3.8cm,1.5cm) arc (-135:-45:1.42cm);
    \draw[color=blue,thick] (5.8cm,1.5cm) -- (6.5cm,2.3cm);
  \end{tikzpicture}
\end{center}
Each of these is an $I_1$ singularity and we have completely split the $I_0^*$ singularity along $s = 0\,.$
The smooth fiber $E_{0,t}$ for a fixed $t$ and $s = 0\,,$ minus the point at infinity i.e. minus $\sigma(0,t)\,,$ is the double cover of $\C$ branched along the roots of
$x^3 + t\epsilon = 0\,.$

Consider first the three cube roots of $-{\epsilon\over a}$ and let us assume that $\epsilon, \ a$  are real, $\epsilon > 0\,,$  $ a < 0\,.$ Fix $ t\ \neq 0$ and consider the (real) plane $\C$ in the variable $s$. Let $\delta(r)\,,$ $ 0 \le r \le 1\,,$ be the straight line path from the origin to the real cube root. Let $\zeta = e^{2\pi i \over 3}\,.$ Then $\zeta\delta(r)$ and $\zeta^2\delta(r)$ will the the paths to the other two roots. Then, if $\rho_1(r),\rho_2(r),\rho_3(r)$ are the roots of
\begin{equation}\label{root1} x^3 - 3c^2(\delta(r))^2x + 2c^3(\delta(r))^3 + at(\delta(r))^3 + t\epsilon = 0\,,\end{equation}
the roots of
\begin{equation}\label{root2} x^3 - 3c^2(\zeta^k\delta(r))^2x + 2c^3(\zeta^k\delta(r))^3 + at(\zeta^k\delta(r))^3 + t\epsilon = 0\,,\end{equation}
will be $\zeta^k\rho_1(r),\zeta^k\rho_2(r),\zeta^k\rho_3(r)\,.$ For example, if $t$ is a real positive number, then all along the path $\delta(r)$ the real root, say  it is $\rho_1(0)\,,$ remains real for $0 \le r \le 1\,$ whereas the complex roots
$\rho_2(0)$ and $\rho_3(0) = \overline{\rho_2(0)}$ remain complex until $r = 1\,,$ at which point they merge into a real root of multiplicity two. Then, along the path $\zeta\delta(r)\,,$ the roots $\zeta\rho_2(0) = \rho_3(0)$ and $\zeta\rho_3(0) = \rho_1(0)$ will merge while $\rho_2(r)$ remains disjoint from the paths taken by the other roots. Similarly along $\zeta^2\delta(r)\,,$ the roots $\rho_1(0)$ and $\rho_2(0)$ will merge. The paths of the merging roots of equation (\ref{root2}) will be the path of the merging roots of  equation (\ref{root1}) multiplied by $\zeta$ or $\zeta^2\,.$ If we assume $2 c^3 > -at\,,$ then it is not hard to see that the real root $\rho_1(s)$ decreases with $s\,.$ Therefore, since $\rho_1(s) + \rho_2(s) + \rho_3(s) = 3c^2(\delta(r))^2$ increases with $s\,,$ so does the real part of the complex roots. Therefore the path of the merging roots never crosses the line joining the complex roots of
$x^3 + t\epsilon = 0\,.$ It follows from the algorithm above or a simple direct argument that the vanishing cycle corresponding to the real cube root of $-{\epsilon\over a}$ will be a homology class $Z \in H_1(E_0)$  represented by the curve lying over this straight line.

The resulting determination of the vanishing cycles can be formulated as follows:
Multiplication by $\zeta\,,$ i.e. rotation through $2\pi\over 3\,,$ induces a homeomorphism of the smooth fiber $E_0\,,$ viewed as the double branched cover of $\bold P^1$ along the roots of $x^3 + t\epsilon = 0$ and ``infinity", and hence an isomorphism $\zeta_*:H_1(E_0) \to H_1(E_0)\,.$  The vanishing cycles corresponding the cube roots of $-{\epsilon\over a}$ will be $\{Z, \zeta_*Z, \zeta_*^2Z\}\,.$

Now suppose that we choose $c$ real so that $at + 4c^3 > 0\,.$ The same argument shows that we get the same vanishing cycles for these points also. When we order all six points by increasing argument in the complex plane, we therefore obtain as our ordered set of vanishing cycles $\{Z,\zeta_*^2Z,\zeta_*Z,Z,\zeta_*^2Z,\zeta_*Z\}\,.$ Note that, in the usual counterclockwise orientation, $Z\cdot \zeta Z = \zeta Z\cdot\zeta^2 Z = \zeta^2Z\cdot Z = 1\,.$ From the relation $1 + \zeta + \zeta^2 = 0\,,$ it follows that the set $\{J\,|\,a(J) = 0\}$ of junctions with vanishing asymptotic charge has dimension four.  The elements
\begin{equation} \alpha_1 = (0,0,0,-1,-1,-1)\,\,\,\alpha_2 = (0,0, -1,0,0,1)\,\,\,\alpha_3= (0,-1,0,0,1,0) \,\,\, \alpha_4 = (-1,0,1,-1,0,1)\end{equation}
are a basis and satisfy
\begin{align}&\langle \alpha_1,\alpha_1\rangle = \langle \alpha_2,\alpha_2\rangle = \langle \alpha_3,\alpha_3\rangle = \langle \alpha_4,\alpha_4\rangle = -2\cr&
\langle \alpha_1,\alpha_2\rangle = \langle \alpha_1,\alpha_3\rangle = \langle \alpha_1,\alpha_4\rangle = 1\cr&
\langle \alpha_2,\alpha_3\rangle = \langle \alpha_2,\alpha_4\rangle = \langle \alpha_
3,\alpha_4\rangle = 0\,.\end{align}
Thus we get the ${\mathfrak d_4}$ Dynkin diagram, and in fact this set of junctions gives the root lattice and weight structure of the ${\mathfrak d_4}$ Lie algebra; see \cite{GrassiHalversonShaneson2013} for more on the Lie algebra details given the geometric junctions.

Finally, we determine the monodromy around the component  $t = 0$ of the singular locus to exhibit the collapse of the ${\mathfrak d_4}$ algebra (in physics language the ``gauge algebra" ${\mathfrak d_4}$) to a ${\mathfrak g_2}$ algebra (``gauge algebra" in physics language). Instead of only $t$ real, take $t(\theta) = te^{i\theta}\,.$ For $t$ small enough, the cube roots of $-{\epsilon t(\theta)\over at + 4c^3}\,$ will be  closer to the origin than those of $-{\epsilon\over a}$ and will rotate clockwise in an almost circular motion around as $\theta$ goes from zero to $2\pi\,.$ When $\theta$ gets to $2\pi$ the roots will have been permuted by multiplication by $\zeta\,.$ For each $\theta\,,$ let $E_0(\theta)$ be the singular fiber over the origin, the branched cover of $\bold C$ along the roots of $x^3 + t(\theta)\epsilon = 0$ and infinity . Then multiplication by $e^{i\theta}$ induces ${(e^{i\theta})}_* : H_1(E_0)\to H_1(E_0(\theta))\,.$ Since the lines from the origin to these points on the discriminant will also rotate around with $\theta\,,$ only changing length slightly, it is clear that the vanishing cycles of the cube roots of $-{\epsilon t(\theta)\over at + 4c^3}$ will be $\{{(e^{i\theta\over 3})}_*Z, \ {(e^{i\theta\over 3})}_*\zeta _*Z, \  {(e^{i\theta\over 3})}_*\zeta_* ^2Z\}\,.$ Therefore the effect of $\theta$ going from $0$ to $2\pi$ is that the vanishing cycles do not change, but the three points on the discriminant locus undergo a rotation though $2\pi\over 3$ i.e. the order of all six vanishing cycles will have changed as the other three do not move.

In fact, as the three cube roots of $-{\epsilon\over a}$ do not move as $\theta$ changes, the ordered sets of pairs of roots that coalesce as we move out from zero to any of these cube roots will have be permuted cyclically each time $\theta$ goes around through $2\pi\,.$  (For each discriminant point and one value of $\theta\,,$ a third root will cross one of these at a value of $r$ less than one.) Thus,  when we get to $\theta = 2\pi\,,$ these vanishing cycles will have moved to  $\{\zeta_*Z, \zeta_*^2 Z, Z\}\,.$ In other words, the effect of $\theta$ going from zero to $2\pi$ will be, since these vanishing cycles moved and the other discriminant points rotated,

\begin{equation}\{Z,\zeta_*^2Z,\zeta_*Z,Z,\zeta_*^2Z,\zeta_*Z\}\mapsto \{\zeta_*Z,Z,\zeta_*^2 Z, \zeta_*Z,Z,\zeta_*^2 Z\}\,.\end{equation}
In other words, on the junction we have the ``outer monodromy" map $\lambda(a,b,c,a,b,c) = (c,a,b,c,a,b)\,.$
(Note: As $\theta$ rotates around, there will be three points where pairs of points on the discriminant locus are on the same line segments from the origin. This is the reason the "obvious" continuity argument gives the wrong result.)

Now we replace $\alpha_4$ with $\alpha_4' = (-1,0,0,1,0,0).$ The Dynkin diagram of the intersection form on the junctions $\alpha_1,\alpha_2,\alpha_3,\alpha_4'$ is also that of ${\mathfrak d_4}\,,$ and  $\alpha_4'$ is also a root. Clearly $\lambda(\alpha_1) = \alpha_1\,,$ $\lambda(\alpha_2) = \alpha_4'\,,$ and $\lambda(\alpha_4') = \alpha_3\,$ and $\lambda(\alpha_3) = \alpha_1\,.$
Therefore, this monodromy map fixes the root corresponding the the central node of the ${\mathfrak d_4}$ Dynkin diagram and permutes the other three nodes, and when we divide by this action, we get the ${\mathfrak g_2}$ Dynkin diagram. However, $\alpha_4'$ is not a simple root in the ${\mathfrak d_4}$ algebra. Or, to put it other way, if we take these to be simple roots, we do not get the simple root lattice of ${\mathfrak d_4\,.}$

The monodromy does not preserve the  Weyl chamber spanned by $\alpha_1,...,\alpha_4$ but moves it to a different one.  However, since it is of order three and must fix the central node because of invariance of intersection numbers, there are only two possibilities for what the monodromy induces on the Dynkin diagram, either a rotation of the extremal nodes or the identity. The calculation with the non-simple root eliminates the trivial case, and therefore it can be concluded from this argument using a non-simple root that the monodromy reduces the ${\mathfrak d_4}$ root lattice and weight structure, i.e. the Lie algebra, to that of ${\mathfrak g_2}\,.$

The result can also be establishes by considering the full set $J^{(-2)}$. In \cite{GrassiHalversonShaneson2013} we showed
that there are precisely $192$ four element subsets of $J^{(-2)}$ which can serve as simple roots; this number matches the dimension
of the Weyl group of ${\mathfrak d_4}$, as it should. Thus $J^{-2}$ contains the full data of the Lie algebra within it. It is straightforward
to show that the outer monodromy map preserves $J^{(-2)}$ and is not trivial. Thus it  gives an automorphism of the algebra,
and from the junction description  it is of order three; it cannot be trivial on the Dynkin diagram without being so on all of $J^{(-2)}\,.$ Therefore it must induces an action on the Dynkin diagram which reduces $\mathfrak{d}_4$
to $\mathfrak{g}_2$.

For further details on the root structure and the full representation theory  see \cite{GrassiHalversonShaneson2013}.

\bibliographystyle{siam}
\bibliography{mybibliography}

\begin{thebibliography}{10}

\bibitem{Kodaira}
{\sc W.~L. Baily, Jr.}, ed., {\em Kunihiko {K}odaira: collected works. {V}ol.
  {I},{II}, {III}}, Iwanami Shoten, Publishers, Tokyo; Princeton University
  Press, Princeton, N.J., 1975.

\bibitem{BaPeVV84}
{\sc W.~Barth, C.~Peters, and A.~Van~de Ven}, {\em Compact complex surfaces},
  vol.~4 of Ergebnisse der Mathematik und ihrer Grenzgebiete (3) [Results in
  Mathematics and Related Areas (3)], Springer-Verlag, Berlin, 1984.

\bibitem{BonoraSavelli}
{\sc L.~Bonora and R.~Savelli}, {\em Non-simply laced {L}ie algebras via {F}
  theory strings}, J. High Energy Phys.,  (2010), pp.~025, 25.

\bibitem{Bredon}
{\sc G.~E. Bredon}, {\em Topology and geometry}, vol.~139 of Graduate Texts in
  Mathematics, Springer-Verlag, New York, 1997.
\newblock Corrected third printing of the 1993 original.

\bibitem{coxeter:weyl}
{\sc H.~S.~M. Coxeter}, {\em Discrete groups generated by reflections}, Ann. of
  Math. (2), 35 (1934), pp.~588--621.

\bibitem{DeWolfeZwiebach}
{\sc O.~DeWolfe and B.~Zwiebach}, {\em String junctions for arbitrary
  {L}ie-algebra representations}, Nuclear Phys. B, 541 (1999), pp.~509--565.

\bibitem{DouglasParkSchnell}
{\sc M.~Douglas, D.~Park, and C.~Schnell}, {\em The {C}remmer-{S}cherk
  {M}echanism in {F}-theory compactifications on {K}3 manifolds}.
\newblock {\tt arXiv::1403.1595}.

\bibitem{DV}
{\sc P.~Du~Val}, {\em Homographies, quaternions and rotations}, Oxford
  Mathematical Monographs, Clarendon Press, Oxford, 1964.

\bibitem{EilenbergSteenrod}
{\sc S.~Eilenberg and N.~Steenrod}, {\em Foundations of algebraic topology},
  Princeton University Press, Princeton, New Jersey, 1952.

\bibitem{Fox-Milnor}
{\sc R.~H. Fox and J.~W. Milnor}, {\em Singularities of {$2$}-spheres in
  {$4$}-space and cobordism of knots}, Osaka J. Math., 3 (1966), pp.~257--267.

\bibitem{GaberdielZwiebach}
{\sc M.~R. Gaberdiel and B.~Zwiebach}, {\em Exceptional groups from open
  strings}, Nuclear Phys. B, 518 (1998), pp.~151--172.

\bibitem{GrassiHalversonShaneson:TerminalBraids}
{\sc A.~Grassi, J.~Halverson, and J.~Shaneson}, {\em Deformation of higher
  dimensional elliptic fibrations}.

\bibitem{GrassiHalversonShaneson:Math}
\leavevmode\vrule height 2pt depth -1.6pt width 23pt, {\em Resolution and
  deformation of elliptic fibrations}.

\bibitem{GrassiHalversonShanesonTaylor2014}
{\sc A.~Grassi, J.~Halverson, J.~Shaneson, and W.~Taylor}, {\em Non-{H}iggsable
  {QCD} and the standard model spectrum in {F}-theory}, J. High Energy Phys.,
  (2015), pp.~086, front matter+57.

\bibitem{GrassiHalversonShaneson2013}
{\sc A.~Grassi, J.~Halverson, and J.~L. Shaneson}, {\em Matter from geometry
  without resolution}, J. High Energy Phys.,  (2013), pp.~1--47.
\newblock {\tt arXiv:1306.1832}.

\bibitem{GrassiHalversonShaneson2014}
{\sc A.~Grassi, J.~Halverson, and J.~L. Shaneson}, {\em Non-abelian gauge
  symmetry and the {H}iggs mechanism in {F}-theory}, Comm. Math. Phys., 336
  (2015), pp.~1231--1257.

\bibitem{Hartshorne}
{\sc R.~Hartshorne}, {\em Algebraic geometry}, Springer-Verlag, New York, 1977.
\newblock Graduate Texts in Mathematics, No. 52.

\bibitem{Hatcher}
{\sc A.~Hatcher}, {\em Algebraic topology}, Cambridge University Press,
  Cambridge, 2002.

\bibitem{McCleary}
{\sc J.~McCleary}, {\em A user's guide to spectral sequences}, vol.~58 of
  Cambridge Studies in Advanced Mathematics, Cambridge University Press,
  Cambridge, second~ed., 2001.

\bibitem{Mi81}
{\sc R.~Miranda}, {\em Smooth models for elliptic threefolds}, in The
  birational geometry of degenerations (Cambridge, Mass., 1981), vol.~29 of
  Progr. Math., Birkh\"auser Boston, Mass., 1983, pp.~85--133.

\bibitem{MorrisonTaylor2012NHC}
{\sc D.~Morrison and W.~Taylor}, {\em Classifying bases for 6d {F}-theory
  models}.
\newblock {\tt arXiv:1201.1943}.

\bibitem{MV}
{\sc D.~R. Morrison and C.~Vafa}, {\em Compactifications of {$F$}-theory on
  {C}alabi-{Y}au threefolds. {I}, ii}, Nuclear Phys. B, 473, 476 (1996),
  pp.~74--92, 437--469.

\bibitem{MVII}
\leavevmode\vrule height 2pt depth -1.6pt width 23pt, {\em Compactifications of
  {$F$}-theory on {C}alabi-{Y}au threefolds. {II}}, Nuclear Phys. B, 476
  (1996), pp.~437--469.

\bibitem{Na88}
{\sc N.~Nakayama}, {\em On {W}eierstrass models}, in Algebraic geometry and
  commutative algebra, Vol.\ II, Kinokuniya, Tokyo, 1988, pp.~405--431.

\bibitem{JimPackage}
{\sc py~junctions}, {\em {A} {C}omputational {P}ackage for {J}unctions
  {U}tilizing {SAGE} and {P}ython}.
\newblock {\tt http://www.jhhalverson.com/deformations/}.

\bibitem{Spanier}
{\sc E.~H. Spanier}, {\em Algebraic topology}, Springer-Verlag, New
  York-Berlin, 1981.
\newblock Corrected reprint.

\bibitem{Steenrod1957}
{\sc N.~Steenrod}, {\em The topology of fibre bundles}, Princeton Landmarks in
  Mathematics, Princeton University Press, Princeton, NJ, 1999.
\newblock Reprint of the 1957 edition, Princeton Paperbacks.

\bibitem{sage}
{\sc W.~Stein}, {\em {SAGE} {M}athematics {S}oftware ({V}ersion 2.7)}, The
  SAGE~Group, 2007.
\newblock {\tt http://www.sagemath.org}.

\bibitem{V1}
{\sc C.~Vafa}, {\em Evidence for {$F$}-theory}, Nuclear Phys. B, 469 (1996),
  pp.~403--415.

\end{thebibliography}

\end{document}